\documentclass[12pt, reqno]{amsart}
\usepackage[T1]{fontenc}
\usepackage{amsfonts}
\usepackage{amsmath}
\usepackage{amsthm}
\usepackage{amssymb}
\usepackage{geometry}
\usepackage{graphicx}
\usepackage{xcolor}
\usepackage{mathtools}
\usepackage[colorlinks=true, linkcolor=blue]{hyperref}
\usepackage{cleveref}
\usepackage[english]{babel}
\usepackage{lineno}
\usepackage{float}
\usepackage{setspace}

\newtheorem{theorem}{Theorem}[section]

\newtheorem{lemma}[theorem]{Lemma}
\newtheorem{conjecture}[theorem]{Conjecture}
\newtheorem{proposition}[theorem]{Proposition}

\usepackage{tikz}
\usetikzlibrary{positioning}
\usetikzlibrary{decorations,arrows}
\usetikzlibrary{decorations.markings}
\numberwithin{equation}{section}

\usepackage{rustic}
\usepackage[T1]{fontenc}

\emergencystretch=\maxdimen
\title{Strong hub cover pebbling number}
\author{Runze Wang}
\address[]{Department of Mathematical Sciences, University of Memphis, Memphis, TN 38152, USA}
\email{rwang6@memphis.edu; runze.w@hotmail.com}
\thanks{}
\date{\today}
\subjclass[2020]{05C57}

\begin{document}

\sloppy

\begin{abstract}
    In a graph $G$, we define a set of vertices to be a \emph{strong hub set} if for any two vertices in $G$, we can find a path between them whose internal vertices are all in this set. We define the \emph{strong hub cover pebbling number} of $G$, denoted by $h_s^*(G)$, to be the smallest $t$ such that for any initial configuration with $t$ pebbles on $G$, we can make some pebbling moves (a pebbling move consists of removing two pebbles from a vertex $v$ and adding one pebble to another vertex adjacent to $v$) so that there is a strong hub set with every vertex in it having a pebble. We determine the strong hub cover pebbling numbers of paths, stars, and books.
\end{abstract}
\keywords{Pebbling number; Hub set; Path; Star; Book}

\maketitle
\doublespacing

\section{Introduction}
In this paper, we only study finite simple connected graphs.

The concept of graph pebbling was first proposed by Lagarias and Saks for giving an alternative proof of a theorem in number theory \cite{Hu}. Then, in 1989, Chung formally introduced this concept and defined the \emph{pebbling number} of a graph in \cite{Ch}. In a graph $G=(V,\, E)$, we distribute $t$ pebbles to the vertices. For two vertices $u$ and $v$ adjacent to each other, assuming that there are at least two pebbles on $u$, we can make a \emph{pebbling move} by removing two pebbles from $u$ and adding one pebble to $v$. The \emph{pebbling number} of $G$, denoted by $\pi(G)$, is the smallest $t$ such that for any initial configuration with $t$ pebbles on $G$ and any chosen \emph{root vertex} $v\in V$, we can make some pebbling moves so that there is a pebble on $v$. Graph pebbling has been extensively studied since it was introduced. For some recent developments, see \cite{ADHS,AH,Fi,KSW}.

There are also some variants of graph pebbling. Isaak and Prudente defined a two-player pebbling game in \cite{IP}. (Also, see \cite{IPPFM,Pr}.) Crull et al. introduced in \cite{CCFHPST} the cover pebbling number of $G$, which is the smallest number $t$ such that for any initial configuration with $t$ pebbles on $G$, we can make some pebbling moves so that every vertex in $G$ has a pebble on it. In \cite{HM}, Hurlbert and Munyan determined the cover pebbling numbers of the hypercubes. Lourdusamy and Tharani introduced the covering cover pebbling number in \cite{LT}. This concept has been studied on different graphs in \cite{LM,LP,Mc}.

Recently, Lourdusamy et al. introduced the \emph{hub cover pubbling number} in \cite{LBP}. The concept of hub sets was introduced by Walsh in \cite{Wa}. In graph $G$, a nonempty vertex set $U\subseteq V$ is called a \emph{hub set} if for any two vertices $v_1,\, v_2\notin U$, we can find a path between $v_1$ and $v_2$ whose internal vertices are all in $U$. (If $v_1$ and $v_2$ are adjacent, then there is a path between them which does not have any internal vertices, but the empty set is also a subset of $U$, so the condition also holds for $v_1$ and $v_2$.) A hub set gives a model, where we convert some buildings to rapid-transit stations (choose some vertices to be in the hub set), so that people can walk from a building to a station, take the rapid-transit system to another station, and walk to the destination building. The \emph{hub cover pebbling number} of $G$, denoted by $h^*(G)$, is the smallest number $t$ such that for any initial configuration with $t$ pebbles on $G$, we can make some pebbling moves so that there is a hub set $U$ with every vertex in $U$ having a pebble. In \cite{LBP}, Lourdusamy et al. determined the hub cover pebbling numbers on some wheel-related graphs.

In the hub set model, one disadvantage is that it is only guaranteed that people can efficiently transit between two buildings (two vertices not in the hub set), but not two stations (two vertices in the hub set) or one station and one building (one vertex in the hub set and one vertex not in the hub set). To solve this problem, we define a nonempty vertex set $U\subseteq V$ to be a \emph{strong hub set} if for any two vertices in the graph, we can find a path between them whose internal vertices are all in $U$. This concept gives us a model where people can conveniently transit between any two places.

Similarly, we define the \emph{strong hub cover pebbling number} of $G$, denoted by $h_s^*(G)$. Note that, a strong hub set is always a hub set, so we have $h_s^*(G)\ge h^*(G)$ for any graph $G$. However, a hub set is not always a strong hub set. For example, as shown in Figure \ref{counter}, in $P_4$, the yellow vertex set $\{v_1,\ v_4\}$ is a hub set but not a strong hub set; in $C_6$, the yellow vertex set $\{v_2,\ v_4,\ v_6\}$ is a hub set but not a strong hub set.

\begin{figure}[H]
    \tikzset{every picture/.style={line width=0.75pt}} 

\begin{tikzpicture}[x=0.75pt,y=0.75pt,yscale=-1,xscale=1]

\draw    (125,603.2) -- (195,603.2) ;
\draw    (195,603.2) -- (265,603.2) ;
\draw    (265,603.2) -- (335,603.2) ;
\draw  [fill={rgb, 255:red, 0; green, 0; blue, 0 }  ,fill opacity=1 ] (259.5,603.2) .. controls (259.5,600.16) and (261.96,597.7) .. (265,597.7) .. controls (268.04,597.7) and (270.5,600.16) .. (270.5,603.2) .. controls (270.5,606.24) and (268.04,608.7) .. (265,608.7) .. controls (261.96,608.7) and (259.5,606.24) .. (259.5,603.2) -- cycle ;
\draw  [fill={rgb, 255:red, 0; green, 0; blue, 0 }  ,fill opacity=1 ] (189.5,603.2) .. controls (189.5,600.16) and (191.96,597.7) .. (195,597.7) .. controls (198.04,597.7) and (200.5,600.16) .. (200.5,603.2) .. controls (200.5,606.24) and (198.04,608.7) .. (195,608.7) .. controls (191.96,608.7) and (189.5,606.24) .. (189.5,603.2) -- cycle ;
\draw  [fill={rgb, 255:red, 245; green, 166; blue, 35 }  ,fill opacity=1 ] (119.5,603.2) .. controls (119.5,600.16) and (121.96,597.7) .. (125,597.7) .. controls (128.04,597.7) and (130.5,600.16) .. (130.5,603.2) .. controls (130.5,606.24) and (128.04,608.7) .. (125,608.7) .. controls (121.96,608.7) and (119.5,606.24) .. (119.5,603.2) -- cycle ;
\draw   (571.2,603.2) -- (534.4,666.94) -- (460.8,666.94) -- (424,603.2) -- (460.8,539.46) -- (534.4,539.46) -- cycle ;
\draw  [fill={rgb, 255:red, 0; green, 0; blue, 0 }  ,fill opacity=1 ] (455.3,539.46) .. controls (455.3,536.42) and (457.76,533.96) .. (460.8,533.96) .. controls (463.84,533.96) and (466.3,536.42) .. (466.3,539.46) .. controls (466.3,542.5) and (463.84,544.96) .. (460.8,544.96) .. controls (457.76,544.96) and (455.3,542.5) .. (455.3,539.46) -- cycle ;
\draw  [fill={rgb, 255:red, 0; green, 0; blue, 0 }  ,fill opacity=1 ] (455.3,666.94) .. controls (455.3,663.9) and (457.76,661.44) .. (460.8,661.44) .. controls (463.84,661.44) and (466.3,663.9) .. (466.3,666.94) .. controls (466.3,669.98) and (463.84,672.44) .. (460.8,672.44) .. controls (457.76,672.44) and (455.3,669.98) .. (455.3,666.94) -- cycle ;
\draw  [fill={rgb, 255:red, 0; green, 0; blue, 0 }  ,fill opacity=1 ] (565.7,603.2) .. controls (565.7,600.16) and (568.16,597.7) .. (571.2,597.7) .. controls (574.24,597.7) and (576.7,600.16) .. (576.7,603.2) .. controls (576.7,606.24) and (574.24,608.7) .. (571.2,608.7) .. controls (568.16,608.7) and (565.7,606.24) .. (565.7,603.2) -- cycle ;
\draw  [fill={rgb, 255:red, 245; green, 166; blue, 35 }  ,fill opacity=1 ] (418.5,603.2) .. controls (418.5,600.16) and (420.96,597.7) .. (424,597.7) .. controls (427.04,597.7) and (429.5,600.16) .. (429.5,603.2) .. controls (429.5,606.24) and (427.04,608.7) .. (424,608.7) .. controls (420.96,608.7) and (418.5,606.24) .. (418.5,603.2) -- cycle ;
\draw  [fill={rgb, 255:red, 245; green, 166; blue, 35 }  ,fill opacity=1 ] (528.9,666.94) .. controls (528.9,663.9) and (531.36,661.44) .. (534.4,661.44) .. controls (537.44,661.44) and (539.9,663.9) .. (539.9,666.94) .. controls (539.9,669.98) and (537.44,672.44) .. (534.4,672.44) .. controls (531.36,672.44) and (528.9,669.98) .. (528.9,666.94) -- cycle ;
\draw  [fill={rgb, 255:red, 245; green, 166; blue, 35 }  ,fill opacity=1 ] (528.9,539.46) .. controls (528.9,536.42) and (531.36,533.96) .. (534.4,533.96) .. controls (537.44,533.96) and (539.9,536.42) .. (539.9,539.46) .. controls (539.9,542.5) and (537.44,544.96) .. (534.4,544.96) .. controls (531.36,544.96) and (528.9,542.5) .. (528.9,539.46) -- cycle ;
\draw  [fill={rgb, 255:red, 245; green, 166; blue, 35 }  ,fill opacity=1 ] (329.5,603.2) .. controls (329.5,600.16) and (331.96,597.7) .. (335,597.7) .. controls (338.04,597.7) and (340.5,600.16) .. (340.5,603.2) .. controls (340.5,606.24) and (338.04,608.7) .. (335,608.7) .. controls (331.96,608.7) and (329.5,606.24) .. (329.5,603.2) -- cycle ;

\draw (116.6,614) node [anchor=north west][inner sep=0.75pt]   [align=left] {$\displaystyle v_{1}$};
\draw (185.6,614) node [anchor=north west][inner sep=0.75pt]   [align=left] {$\displaystyle v_{2}$};
\draw (256.6,614) node [anchor=north west][inner sep=0.75pt]   [align=left] {$\displaystyle v_{3}$};
\draw (325.6,614) node [anchor=north west][inner sep=0.75pt]   [align=left] {$\displaystyle v_{4}$};
\draw (451.6,520) node [anchor=north west][inner sep=0.75pt]   [align=left] {$\displaystyle v_{1}$};
\draw (525.6,520) node [anchor=north west][inner sep=0.75pt]   [align=left] {$\displaystyle v_{6}$};
\draw (451.6,676) node [anchor=north west][inner sep=0.75pt]   [align=left] {$\displaystyle v_{3}$};
\draw (525.6,676) node [anchor=north west][inner sep=0.75pt]   [align=left] {$\displaystyle v_{4}$};
\draw (398.6,597) node [anchor=north west][inner sep=0.75pt]   [align=left] {$\displaystyle v_{2}$};
\draw (579.6,597) node [anchor=north west][inner sep=0.75pt]   [align=left] {$\displaystyle v_{5}$};

\end{tikzpicture}
\caption{Two examples of hub sets which are not strong hub sets.}
\label{counter}
\end{figure}
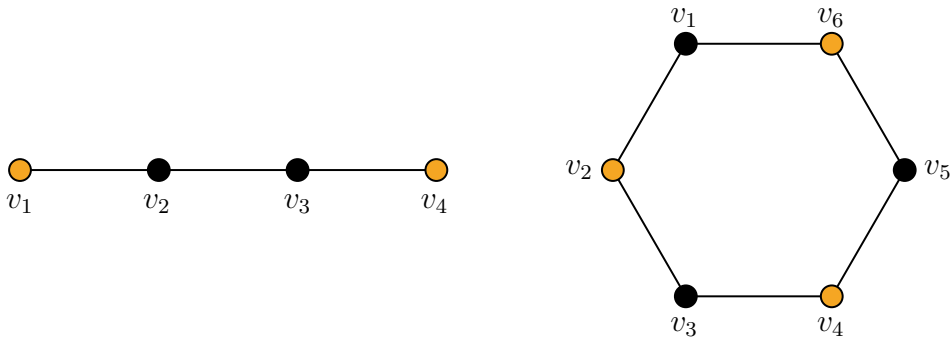

Also, we would like to mention that, a strong hub set is always a dominating set, but a dominating set is not always a strong hub set.

In this short paper, we determine the strong hub cover pebbling numbers of paths, stars, and books. Also, we suggest an open problem about cycles.

\section{Paths}
We denote the path on $n$ vertices by $P_n$ and denote the $n$ vertices from one end to the other by $v_1,\ v_2,\ ...\ v_n$. Trivially, we have $h_s^*(P_1)=h_s^*(P_2)=0$.
\begin{theorem}
    For any $n\ge 3$, we have
    \begin{align*}
        h_s^*(P_n)=2^{n-1}-1.
    \end{align*}
\end{theorem}

We observe that for $n\ge 3$, a vertex set in $P_n$ is a strong hub set if and only if it contains the $n-2$ vertices in the middle, so we have to let each of $v_2,\ v_3,\ ...\ v_{n-1}$ have a pebble after some pebbling moves. If we only have $2^{n-1}-2$ pebbles, then we can put $2^{n-1}-3$ of them on $v_1$ and put one of them on $v_n$, and it is easy to see that we cannot get a pebble on $v_{n-1}$ after any pebbling moves. So $h_s^*(P_n)\ge 2^{n-1}-1$.

To show that $2^{n-1}-1$ pebbles are enough, we will prove a slightly stronger result by induction.

\begin{proposition}
    If we have $2^{n-1}-1$ pebbles, then after some pebbling moves, we can get (at least) one pebble on each of $v_2,\ v_3,\ ...,\ v_{n-1}$, and also get (at least) one pebble on $v_1$ or $v_n$.
\end{proposition}

\begin{proof}
    For $P_3$, assume that we have three pebbles. If both $v_1$ and $v_2$ have pebble(s) or both $v_2$ and $v_3$ have pebble(s), then the proof is complete. So, by symmetry, we need to handle the following three cases.
    \begin{itemize}
        \item $v_1$ has two pebbles and $v_3$ has one pebble.
        \item $v_1$ has three pebbles.
        \item $v_2$ has three pebbles.
    \end{itemize}
    In the first case, we can make a pebbling move by removing two pebbles from $v_1$ and adding a pebble to $v_2$. In the second case, we can also remove two pebbles from $v_1$ and add a pebble to $v_2$. In the third case, we can remove two pebbles from $v_2$ and add a pebble to $v_1$. So the base case $n=3$ is proved.

    Assume that the conclusion holds true for $n=k-1$. Assume $n=k$ and we have $2^{k-1}-1$ pebbles. By the pigeonhole principle, we know that there are at least $2^{k-2}$ pebbles on $\{v_1,\ v_2,\ ...,\ v_{k-1}\}$ or $\{v_2,\ v_3,\ ...,\ v_k\}$. By symmetry, we assume that there are at least $2^{k-2}$ pebbles on $\{v_1,\ v_2,\ ...,\ v_{k-1}\}$. By the induction hypothesis, we can use at most $2^{k-2}-1$ pebbles to make some pebbling moves to get a pebble on each of $v_2,\ v_3,\ ...,\ v_{k-2}$ and also get a pebble on $v_1$ or $v_{k-1}$.
    
    We have used at most $2^{k-2}-1$ pebbles, so there are at least $2^{k-2}$ pebbles not used yet. We have two cases.

    \textbf{Case 1.} There is a pebble on $v_1$. In this case, we need to get a pebble on $v_{k-1}$.

    \textbf{Subcase 1.1.} There are at least two pebbles on $v_k$. In this subcase, we just need to remove two pebbles from $v_k$ and add a pebble to $v_{k-1}$.

    \textbf{Subcase 1.2.} There is exactly one pebble on $v_k$. In this subcase, we still have at least $2^{k-2}-1$ pebbles not used yet, and all of them are on $\{v_1,\ v_2,\ ...,\ v_{k-1}\}$. We also take an extra pebble from $v_1$, so there are $2^{k-2}$ pebbles to use. We have the following lemma (see \cite{Hu}) about the pebbling number of a path. 
    
    \begin{lemma}\label{lemma}
        $\pi(P_m)=2^{m-1}$.
    \end{lemma}
    Now, $v_1,\ v_2,\ ...,\ v_{k-1}$ form a $P_{k-1}$, so we can use the $2^{k-2}$ pebbles to make some pebbling moves and get a pebble on $v_{k-1}$. Note that, in order to have $2^{k-2}$ pebbles to use, we took a pebble from $v_1$, so if $v_1$ only had one pebble on it before these pebbling moves, then after the pebbling moves, it is possible that there is no pebble on $v_1$. But in this subcase, we have assumed that there is a pebble on $v_k$. So now we have a pebble on each of $v_2,\ v_3,\ ...,\ v_k$, as desired.

    \textbf{Subcase 1.3.} There are no pebbles on $v_k$. In this subcase, we have at least $2^{k-2}$ unused pebbles, all of which are on $\{v_1,\ v_2,\ ...,\ v_{k-1}\}$. So by Lemma \ref{lemma}, we can make some pebbling moves to get a pebble on $v_{k-1}$.

    \textbf{Case 2.} There is no pebble on $v_1$, but there is a pebble on $v_{k-1}$. In this case, we need to get a pebble on $v_1$ or $v_k$. If there is already a pebble on $v_k$, then the proof is complete. If $v_k$ does not have a pebble, then there are at least $2^{k-2}$ unused pebbles, all of which are on $\{v_1,\ v_2,\ ...,\ v_{k-1}\}$. By Lemma \ref{lemma}, we can make some pebbling moves to get a pebble on $v_1$.
\end{proof}

\section{Stars and books}
For $n\ge 2$, the \emph{star} on $n+1$ vertices, denoted by $S_n$, is constructed by connecting $n$ independent vertices to a universal vertex. This universal vertex will be called the \emph{center} of the star. It is easy to see that a vertex set in a star is a strong hub set if and only if it contains the center.

In $S_n$, if we only have $n$ pebbles, then we can put one pebble on each vertex which is not the center. We cannot get a pebble on the center by making pebbling moves, so $h_s^*(S_n)\ge n+1$. If we have at least $n+1$ pebbles and none of them is on the center, then by the pigeonhole principle, there must be a vertex with two pebbles on it. We can remove these two pebbles from this vertex and place one pebble on the center. So $n+1$ pebbles are enough.

\begin{proposition}\label{star}
    For any $n\ge 2$, we have
    \begin{align*}
        h_s^*(S_n)=n+1.
    \end{align*}
\end{proposition}

Then, for $n\ge 2$, the \emph{book} on $2n+2$ vertices, denoted by $B_n$, is defined by $S_n\square P_2$, the Cartesian product of the star on $n+1$ vertices and the path on two vertices.

\begin{theorem}
    For any $n\ge 2$, we have
    \begin{align*}
        h_s^*(B_n)=2n+3.
    \end{align*}
\end{theorem}

\begin{proof}
We construct $B_n$ in the following way.
\begin{itemize}
    \item Draw two copies of $S_n$: One has center $a$ and $n$ vertices $u_1,\ u_2,\ ...,\ u_n$ connected to $a$; the other has center $b$ and $n$ vertices $v_1,\ v_2,\ ...,\ v_n$ connected to $b$.
    \item Connect $a$ with $b$ and connect $u_i$ with $v_i$ for any $1\le i\le n$.
\end{itemize}

An example of $B_6$ is shown in Figure \ref{B6}.

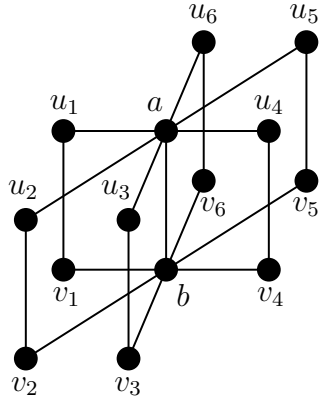
\begin{figure}[H]
    \tikzset{every picture/.style={line width=0.75pt}} 

\begin{tikzpicture}[x=0.75pt,y=0.75pt,yscale=-1,xscale=1]

\draw    (315.15,68.54) -- (277.23,158.26) ;
\draw    (244.39,113.4) -- (347.99,113.4) ;
\draw    (225.43,158.26) -- (366.95,68.54) ;
\draw  [fill={rgb, 255:red, 0; green, 0; blue, 0 }  ,fill opacity=1 ] (290.69,113.4) .. controls (290.69,110.36) and (293.16,107.9) .. (296.19,107.9) .. controls (299.23,107.9) and (301.69,110.36) .. (301.69,113.4) .. controls (301.69,116.44) and (299.23,118.9) .. (296.19,118.9) .. controls (293.16,118.9) and (290.69,116.44) .. (290.69,113.4) -- cycle ;
\draw  [fill={rgb, 255:red, 0; green, 0; blue, 0 }  ,fill opacity=1 ] (238.89,113.4) .. controls (238.89,110.36) and (241.36,107.9) .. (244.39,107.9) .. controls (247.43,107.9) and (249.89,110.36) .. (249.89,113.4) .. controls (249.89,116.44) and (247.43,118.9) .. (244.39,118.9) .. controls (241.36,118.9) and (238.89,116.44) .. (238.89,113.4) -- cycle ;
\draw  [fill={rgb, 255:red, 0; green, 0; blue, 0 }  ,fill opacity=1 ] (219.93,158.26) .. controls (219.93,155.22) and (222.4,152.76) .. (225.43,152.76) .. controls (228.47,152.76) and (230.93,155.22) .. (230.93,158.26) .. controls (230.93,161.3) and (228.47,163.76) .. (225.43,163.76) .. controls (222.4,163.76) and (219.93,161.3) .. (219.93,158.26) -- cycle ;
\draw  [fill={rgb, 255:red, 0; green, 0; blue, 0 }  ,fill opacity=1 ] (271.73,158.26) .. controls (271.73,155.22) and (274.2,152.76) .. (277.23,152.76) .. controls (280.27,152.76) and (282.73,155.22) .. (282.73,158.26) .. controls (282.73,161.3) and (280.27,163.76) .. (277.23,163.76) .. controls (274.2,163.76) and (271.73,161.3) .. (271.73,158.26) -- cycle ;
\draw  [fill={rgb, 255:red, 0; green, 0; blue, 0 }  ,fill opacity=1 ] (342.49,113.4) .. controls (342.49,110.36) and (344.96,107.9) .. (347.99,107.9) .. controls (351.03,107.9) and (353.49,110.36) .. (353.49,113.4) .. controls (353.49,116.44) and (351.03,118.9) .. (347.99,118.9) .. controls (344.96,118.9) and (342.49,116.44) .. (342.49,113.4) -- cycle ;
\draw  [fill={rgb, 255:red, 0; green, 0; blue, 0 }  ,fill opacity=1 ] (361.45,68.54) .. controls (361.45,65.5) and (363.92,63.04) .. (366.95,63.04) .. controls (369.99,63.04) and (372.45,65.5) .. (372.45,68.54) .. controls (372.45,71.58) and (369.99,74.04) .. (366.95,74.04) .. controls (363.92,74.04) and (361.45,71.58) .. (361.45,68.54) -- cycle ;
\draw  [fill={rgb, 255:red, 0; green, 0; blue, 0 }  ,fill opacity=1 ] (309.65,68.54) .. controls (309.65,65.5) and (312.12,63.04) .. (315.15,63.04) .. controls (318.19,63.04) and (320.65,65.5) .. (320.65,68.54) .. controls (320.65,71.58) and (318.19,74.04) .. (315.15,74.04) .. controls (312.12,74.04) and (309.65,71.58) .. (309.65,68.54) -- cycle ;
\draw    (315.15,138.54) -- (277.23,228.26) ;
\draw    (244.39,183.4) -- (347.99,183.4) ;
\draw    (225.43,228.26) -- (366.95,138.54) ;
\draw  [fill={rgb, 255:red, 0; green, 0; blue, 0 }  ,fill opacity=1 ] (290.69,183.4) .. controls (290.69,180.36) and (293.16,177.9) .. (296.19,177.9) .. controls (299.23,177.9) and (301.69,180.36) .. (301.69,183.4) .. controls (301.69,186.44) and (299.23,188.9) .. (296.19,188.9) .. controls (293.16,188.9) and (290.69,186.44) .. (290.69,183.4) -- cycle ;
\draw  [fill={rgb, 255:red, 0; green, 0; blue, 0 }  ,fill opacity=1 ] (238.89,183.4) .. controls (238.89,180.36) and (241.36,177.9) .. (244.39,177.9) .. controls (247.43,177.9) and (249.89,180.36) .. (249.89,183.4) .. controls (249.89,186.44) and (247.43,188.9) .. (244.39,188.9) .. controls (241.36,188.9) and (238.89,186.44) .. (238.89,183.4) -- cycle ;
\draw  [fill={rgb, 255:red, 0; green, 0; blue, 0 }  ,fill opacity=1 ] (219.93,228.26) .. controls (219.93,225.22) and (222.4,222.76) .. (225.43,222.76) .. controls (228.47,222.76) and (230.93,225.22) .. (230.93,228.26) .. controls (230.93,231.3) and (228.47,233.76) .. (225.43,233.76) .. controls (222.4,233.76) and (219.93,231.3) .. (219.93,228.26) -- cycle ;
\draw  [fill={rgb, 255:red, 0; green, 0; blue, 0 }  ,fill opacity=1 ] (271.73,228.26) .. controls (271.73,225.22) and (274.2,222.76) .. (277.23,222.76) .. controls (280.27,222.76) and (282.73,225.22) .. (282.73,228.26) .. controls (282.73,231.3) and (280.27,233.76) .. (277.23,233.76) .. controls (274.2,233.76) and (271.73,231.3) .. (271.73,228.26) -- cycle ;
\draw  [fill={rgb, 255:red, 0; green, 0; blue, 0 }  ,fill opacity=1 ] (342.49,183.4) .. controls (342.49,180.36) and (344.96,177.9) .. (347.99,177.9) .. controls (351.03,177.9) and (353.49,180.36) .. (353.49,183.4) .. controls (353.49,186.44) and (351.03,188.9) .. (347.99,188.9) .. controls (344.96,188.9) and (342.49,186.44) .. (342.49,183.4) -- cycle ;
\draw  [fill={rgb, 255:red, 0; green, 0; blue, 0 }  ,fill opacity=1 ] (361.45,138.54) .. controls (361.45,135.5) and (363.92,133.04) .. (366.95,133.04) .. controls (369.99,133.04) and (372.45,135.5) .. (372.45,138.54) .. controls (372.45,141.58) and (369.99,144.04) .. (366.95,144.04) .. controls (363.92,144.04) and (361.45,141.58) .. (361.45,138.54) -- cycle ;
\draw  [fill={rgb, 255:red, 0; green, 0; blue, 0 }  ,fill opacity=1 ] (309.65,138.54) .. controls (309.65,135.5) and (312.12,133.04) .. (315.15,133.04) .. controls (318.19,133.04) and (320.65,135.5) .. (320.65,138.54) .. controls (320.65,141.58) and (318.19,144.04) .. (315.15,144.04) .. controls (312.12,144.04) and (309.65,141.58) .. (309.65,138.54) -- cycle ;
\draw    (225.43,158.26) -- (225.43,228.26) ;
\draw    (244.39,113.4) -- (244.39,183.4) ;
\draw    (277.23,158.26) -- (277.23,228.26) ;
\draw    (296.19,113.4) -- (296.19,183.4) ;
\draw    (347.99,113.4) -- (347.99,183.4) ;
\draw    (315.15,68.54) -- (315.15,138.54) ;
\draw    (366.95,68.54) -- (366.95,138.54) ;

\draw (285,95) node [anchor=north west][inner sep=0.75pt]   [align=left] {$\displaystyle a$};
\draw (236,94) node [anchor=north west][inner sep=0.75pt]   [align=left] {$\displaystyle u_{1}$};
\draw (215,137) node [anchor=north west][inner sep=0.75pt]   [align=left] {$\displaystyle u_{2}$};
\draw (262,137) node [anchor=north west][inner sep=0.75pt]   [align=left] {$\displaystyle u_{3}$};
\draw (306,48) node [anchor=north west][inner sep=0.75pt]   [align=left] {$\displaystyle u_{6}$};
\draw (357,48) node [anchor=north west][inner sep=0.75pt]   [align=left] {$\displaystyle u_{5}$};
\draw (339,94) node [anchor=north west][inner sep=0.75pt]   [align=left] {$\displaystyle u_{4}$};
\draw (237.6,191) node [anchor=north west][inner sep=0.75pt]   [align=left] {$\displaystyle v_{1}$};
\draw (216.6,236) node [anchor=north west][inner sep=0.75pt]   [align=left] {$\displaystyle v_{2}$};
\draw (268.6,236) node [anchor=north west][inner sep=0.75pt]   [align=left] {$\displaystyle v_{3}$};
\draw (311.65,145.54) node [anchor=north west][inner sep=0.75pt]   [align=left] {$\displaystyle v_{6}$};
\draw (358.45,145.54) node [anchor=north west][inner sep=0.75pt]   [align=left] {$\displaystyle v_{5}$};
\draw (340.6,191) node [anchor=north west][inner sep=0.75pt]   [align=left] {$\displaystyle v_{4}$};
\draw (300,189) node [anchor=north west][inner sep=0.75pt]   [align=left] {$\displaystyle b$};

\end{tikzpicture}
\caption{An example of $B_6$.}
\label{B6}
\end{figure}

First we prove that a vertex set $S\subseteq V(B_n)$ is a strong hub set if and only if it contains one of the following three sets, which are shown by the yellow vertices in Figure \ref{strong}.
\begin{itemize}
    \item $\{a,\ b\}$;
    \item $\{a,\ u_1,\ u_2,\ ...,\ u_n\}$;
    \item $\{b,\ v_1,\ v_2,\ ...,\ v_n\}$.
\end{itemize}

\begin{figure}[H]
    \input strong.tex
\end{figure}

For the sufficiency, it is easy to see that if $S$ contains one of $\{a,\ b\}$, $\{a,\ u_1,\ u_2,\ ...,\ u_n\}$, and $\{b,\ v_1,\ v_2,\ ...,\ v_n\}$, then $S$ is a strong hub set. For the necessity, assume that $S$ does not contain any one of the three vertex set. Then, one of the following three statements must hold true.
\begin{itemize}
    \item $S\cap \{a,\ b\}=\emptyset$.
    \item $a\in S$, $b\notin S$, and $u_i\notin S$ for some $1\le i\le n$.
    \item $b\in S$, $a\notin S$, and $v_j\notin S$ for some $1\le j\le n$.
\end{itemize}

If $S\cap \{a,\ b\}=\emptyset$, then for $u_1$ and $v_2$, we cannot find a path between them whose internal vertices are all in $S$. If $a\in S$, $b\notin S$, and $u_i\notin S$ for some $1\le i\le n$, then for $a$ and $v_i$, we cannot find a path between them whose internal vertices are all in $S$. If $b\in S$, $a\notin S$, and $v_j\notin S$ for some $1\le j\le n$, then for $b$ and $u_j$, we cannot find a path between them whose internal vertices are all in $S$. Thus, by definition, $S$ is not a strong hub set, and the necessity is also proved.

Henceforth, if after some pebbling moves, there is at least one pebble on each vertex in $\{a,\ b\}$, we will say that "\textbf{condition 1} is fulfilled"; if after some pebbling moves, there is at least one pebble on each vertex in $\{a,\ u_1,\ u_2,\ ...,\ u_n\}$, we will say that "\textbf{condition 2} is fulfilled"; and if after some pebbling moves, there is at least one pebble on each vertex in $\{b,\ v_1,\ v_2,\ ...,\ v_n\}$, we will say that "\textbf{condition 3} is fulfilled".

Now, if we only have $2n+2$ pebbles, then we can put five pebbles on $u_1$, put one pebble on each of $u_2,\ u_3,\ ...,\ u_{n-1}$, and put one pebble on each of $v_1,\ v_2,\ ...,\ v_{n-1}$. It is easy to check that we cannot make pebbling moves so that one of the three conditions is fulfilled. So $h_s^*(B_n)\ge 2n+3$.

Then we show that $2n+3$ pebbles are enough. By symmetry, we need to handle three cases.

\textbf{Case 1.} There are at least $n+1$ pebbles on $\{a,\ u_1,\ u_2,\ ...,\ u_n\}$, and at least $n+1$ pebbles on $\{b,\ v_1,\ v_2,\ ...,\ v_n\}$.

In this case, by Proposition \ref{star}, we know that we can make pebbling moves to get a pebble on $a$ and a pebble on $b$, so condition 1 is fulfilled.

\textbf{Case 2.} There are exactly $n+3$ pebbles on $\{a,\ u_1,\ u_2,\ ...,\ u_n\}$, and exactly $n$ pebbles on $\{b,\ v_1,\ v_2,\ ...,\ v_n\}$. Depending on the pebble distribution in $\{b,\ v_1,\ v_2,\ ...,\ v_n\}$, we have three subcases.

\textbf{Subcase 2.1.} There is a pebble on $b$.

In this subcase, if there is also a pebble on $a$, then condition 1 is fulfilled; otherwise, there must be a $u_i$ with at least two pebbles on it, so we can remove two pebbles from $u_i$ and add a pebble to $a$, and then condition 1 is fulfilled.

\textbf{Subcase 2.2.} There are no pebbles on $b$, but there is a $v_i$ with at least two pebbles on it.

In this subcase, we just need to remove two pebbles from $v_i$ and add a pebble to $b$, and then it becomes Subcase 2.1.

\textbf{Subcase 2.2.} There are no pebbles on $b$, and there is exactly one pebble on each $v_i$.

If there are at least two pebbles on $a$, then we can remove two pebbles from $a$ and add a pebble to $b$, and condition 3 is fulfilled.

If there is exactly one pebble on $a$, then there is a $u_i$ with at least two pebbles on it, so we can remove two pebbles from $u_i$ and add a pebble to $a$. Now there are two pebbles on $a$, so we can remove these two pebbles from $a$ and add a pebble to $b$, and condition 3 is fulfilled.

If there are no pebbles on $a$, then either there is a $u_i$ with at least four pebbles on it, or there are $u_j$ and $u_k$ with at least two pebbles on each of them. In either case we can remove four pebbles (either from $u_i$, or from $u_j$ and $u_k$) and add two pebbles to $a$. Then we can remove these two pebbles from $a$ and add a pebble to $b$, and condition 3 is fulfilled.

\textbf{Case 3.} There are at least $n+4$ pebbles on $\{a,\ u_1,\ u_2,\ ...,\ u_n\}$. Depending on how many pebbles there are on $a$, we have four subcases.

\textbf{Subcase 3.1.} There are at least three pebbles on $a$.

In this subcase, we can remove two pebbles from $a$ and add a pebble to $b$, and then condition 1 is fulfilled.

\textbf{Subcase 3.2.} There are exactly two pebbles on $a$.

In this subcase, there must be a $u_i$ with at least two pebbles on it, so we can remove two pebbles from $u_i$ and add a pebble to $a$. Now there are three pebbles on $a$, and it becomes Subcase 3.1.

\textbf{Subcase 3.3.} There is exactly one pebble on $a$.

In this subcase, there are at least $n+3$ pebbles on $\{u_1,\ u_2,\ ...,\ u_n\}$, so either there is a $u_i$ with at least four pebbles on it, or there are $u_j$ and $u_k$ with at least two pebbles on each of them. In either case we can remove four pebbles (either from $u_i$, or from $u_j$ and $u_k$) and add two pebbles to $a$. Now there are three pebbles on $a$, and it becomes Subcase 3.1.

\textbf{Subcase 3.4.} There are no pebbles on $a$.

If there is at least one pebble on each vertex in $\{u_1,\ u_2,\ ...,\ u_n\}$, then we have two possible cases.
\begin{itemize}
    \item There is a $u_i$ with at least three pebbles on it. We can remove two pebbles from $u_i$ and add one pebble to $a$. Now condition 2 is fulfilled.
    \item There are $u_i,\ u_j,\ u_k,\ u_\ell$ with exactly two pebbles on each of them. We can remove these eight pebbles and add four pebbles to $a$. Then we remove two pebbles from $a$ and add one pebble to $b$. Now condition 1 is fulfilled.
\end{itemize}

If there is a vertex in $\{u_1,\ u_2,\ ...,\ u_n\}$, say $u_1$, with no pebbles on it, then there are $n+4$ pebbles on $\{u_2,\ u_3,\ ...,\ u_n\}$, and we have four possible cases.
\begin{itemize}
    \item There is a $u_i\in \{u_2,\ u_3,\ ...,\ u_n\}$ with at least six pebbles on it. We can remove these six pebbles and add three pebbles to $a$. Then we remove two pebbles from $a$ and add a pebble to $b$. Now condition 1 is fulfilled.
    \item There are $u_i,\ u_j\in \{u_2,\ u_3,\ ...,\ u_n\}$ with exactly two pebbles on $u_i$ and exactly five pebbles on $u_j$. We remove two pebbles from $u_i$ and add one pebble to $a$. Then we remove four pebbles from $u_j$ and add two pebbles to $a$. Now there are three pebbles on $a$, and we can remove two pebbles from $a$ and add a pebble to $b$ to fulfill condition 1.
    \item There are $u_i,\ u_j\in \{u_2,\ u_3,\ ...,\ u_n\}$ with exactly three pebbles on $u_i$ and exactly four pebbles on $u_j$. This is the same as the previous case.
    \item There are $u_i,\ u_j,\ u_k\in \{u_2,\ u_3,\ ...,\ u_n\}$ with at least two pebbles on each of them. We can remove these six pebbles and add three pebbles to $a$. Then we can remove two pebbles from $a$ and add a pebble to $b$ to fulfill condition 1.
\end{itemize}
\end{proof}

\section{An open problem}
We consider the strong hub cover pebbling numbers of cycles. In a cycle $C_n$, we can pick a vertex, label it $v_1$, and counterclockwise label the other $n-1$ vertices $v_2,\ v_3,\ ...,\ v_n$. We observe that, up to isomorphism, a vertex set in $C_n$ is a strong hub set if and only if it contains $v_1,\ v_2,\ ...,\ v_{n-2}$. Here we make a conjecture.
\begin{conjecture}
    \[
    h_s^*(C_n)=\begin{cases}
        2^k+2^{k-1}-3, & if\ n=2k\ is\ even,\\
        2^{k+1}-3, & if\ n=2k+1\ is\ odd.
    \end{cases}
    \]
\end{conjecture}

To see the lower bound: If $n=2k$ is even, and we only have $2^k+2^{k-1}-4$ pebbles, then we can put all of them on $v_1$. Similarly, if $n=2k+1$ is odd, and we only have $2^{k+1}-4$ pebbles, then we can also put all of them on $v_1$. It is easy to verify that in these two initial configurations, we cannot choose a strong hub set $S\subseteq V(C_n)$ and place one pebble on each vertex in $S$ by making pebbling moves.

\end{document}